\documentclass[12pt,a4paper,twoside]{amsart}
\usepackage{a4wide,charter, amsmath,amsbsy,amsfonts,amssymb,
stmaryrd,amsthm,mathrsfs,graphicx, amscd, tikz-cd, cancel}
\usepackage[normalem]{ulem}
\usetikzlibrary{matrix,arrows,decorations.pathmorphing}
\usepackage[active]{srcltx}
\usepackage[
	hypertexnames=false,
	hyperindex,
	pagebackref,
	pdftex,
	breaklinks=true,
	bookmarks=false,
	colorlinks,
	linkcolor=blue,
	citecolor=red,
	urlcolor=red,
]{hyperref}
\usepackage[mathscr]{eucal}

\newcommand\PP{{\mathbb P}}

\newcommand\PGL{\operatorname{PGL}}

\newcommand\aut{\operatorname{Aut}}

\renewcommand\dim{\operatorname{dim}}

\newcommand\hookra{\hookrightarrow}

\newcommand\da{\downarrow}

\renewcommand{\hom}{\operatorname{Hom}}
\newcommand\ext{\operatorname{Ext}}
 
\newcommand\sr{\stackrel}

\newcommand\cM{{\mathscr M}}
\newcommand\cN{{\mathcal N}}
\newcommand\cO{{\mathscr O}}

\newcommand\cC{{\mathscr C}}

\newcommand\cI{{\mathscr I}}

\newcommand\cT{{\mathscr T}}

\def\co{\colon\thinspace}

\newcommand\End{\operatorname{End}}
\newcommand\sym{\operatorname{Sym}}
\newcommand\hilb{\operatorname{Hilb}}

\newcommand\slfrak{\operatorname{\mathfrak{sl}}}
\newcommand\autfrak{\operatorname{\mathfrak{aut}}}

\newtheorem{theorem}{Theorem}[section]
\newtheorem{corollary}[theorem]{Corollary}

\newtheorem{lemma}[theorem]{Lemma}

\newtheorem{question}[theorem]{Question}

\theoremstyle{definition}
   
\newtheorem{remark}[theorem]{Remark}

\begin{document}
\author{Marco Boggi}
\address{Departamento de Matem\'atica, UFMG, Belo Horizonte - MG (Brasil)}
\email{marco.boggi@gmail.com}

\author{Eduard Looijenga}
\address{Yau Mathematical Sciences Center, Tsinghua University Beijing (China) and Mathematisch Instituut, Universiteit Utrecht (Nederland)}
\email{eduard@math.tsinghua.edu.cn}



\title{Deforming a canonical curve inside a quadric}
\begin{abstract}
Let $C\subset\PP^{g-1}$ be a canonically  embedded nonsingular nonhyperelliptic curve of genus $g$ and let  $X\subset\PP^{g-1}$ be a quadric containing $C$. Our main result states  among other things  that the Hilbert scheme of $X$  is at  $[C\subset X]$  a local complete intersection of dimension $g^2-1$, and is smooth when $X$ is.  It also includes the assertion that the  minimal obstruction space for this deformation problem is in fact the full associated $\ext^1$-group and that in particular the deformations of $C$ in $X$ are obstructed in case $C$ meets the singular locus of $X$.
As we will show in a forthcoming paper, this  has applications of a topological nature. 
\end{abstract}

\maketitle
\section{Statement of the main result}
Throughout this paper we work  over an algebraically closed field $\Bbbk$ of characteristic zero.
Let $C$ be a smooth nonhyperelliptic projective curve of genus $g$ (so that $g>2$) and 
regard $C$ as embedded in  $\PP:=\check{\PP}(H^0(C,\Omega_C))$. It is well-known that the Hilbert scheme of $\PP$ is smooth  at $C\subset \PP$
of dimension $3g-3+g^2-1(=\dim \cM_g+\dim \PGL_g)$ and that the canonical embedding is unique modulo the action of the projective linear group $\aut(\PP)$ of $\PP$. It is also known that when $g\ge 4$, $C$ is contained in a quadric hypersurface. Let $X\subset\PP$ be one such quadric.  Our main theorem helps us to understand what conditions are imposed on the  deformation theory of the $C$ in $\PP$ by requiring that $C$ stays inside $X$. 

In order to state it we use the following notions and notation. For a variety $Z$, $\cT_Z$ stands for its Zariski tangent sheaf, i.e., the $\cO_Z$-dual of $\Omega_Z$. If $Y\subseteq Z$ is a closed subscheme defined by the $\cO_Z$-ideal $\cI_Y\subseteq \cO_Z$, then 
$\cC_{Y/Z}:=\cI_Y/\cI_Y^2$ (regarded as an $\cO_Z$-module)  is the \emph{conormal sheaf of $Y$ in $Z$},  and  its $\cO_Z$-dual, denoted $\cN_{Y/Z}$, is the \emph{normal sheaf  of $Y$ in $Z$}. 

\begin{theorem}\label{obstruction2}
Let $C\subset X\subset \PP$ be as above (so $C$ nonhyperelliptic of genus $g\ge 4$). Then:
\begin{enumerate}
\item the Hilbert scheme $\hilb(X)$ is a local complete intersection at $[C\subset X]$ of dimension $g^2-1$ with embedding dimension
$g^2-1+\dim\ext^1_C(\cC_{C/X},\cO_C)$,
\item $\ext^1_C(\cC_{C/X},\cO_C)$ is a minimal obstruction space for deformations of $C$ in $X$ and  
\item when $X$ is nonsingular, $\ext^1_C(\cC_{C/X},\cO_C)=0$ and $\hilb(X)$ is smooth at $[C\subset X]$.
\end{enumerate}
\end{theorem}

We shall also   show (Corollary \ref{singular}) that when the quadric $X$ is singular and $C$ meets its singular locus, the obstruction space 
$\ext^1_C(\cC_{C/X},\cO_C)$ is nonzero.
\\

\emph{Acknowledgement.} We thank Eduardo Sernesi  for a helpful conversation with one of us.
E.L.\ was supported by the Chinese National Science Foundation.

\section{Proof of the theorem}
Let us  write $\autfrak(\PP)$ for the Lie algebra of the projective linear group $\aut (\PP)$ of $\PP$, in other words, the Lie algebra of vector fields on $\PP$.  We first observe:

\begin{lemma}\label{lemma:L1}
The natural map $\autfrak(\PP)\to H^0(C,  \cO_C\otimes \cT_\PP)$ is an
isomorphism, and the cohomology spaces $H^1(C, \cO_C\otimes \cT_\PP)$,  $\ext^1_C(\cC_{C/\PP},\cO_C)= H^1(C,\cN_{C/\PP})$ all vanish (so that the first order deformations of $C$ in $\PP$ are unobstructed). Furthermore,  the sequence
\[
0\to \autfrak(\PP)\to H^0(C,\cN_{C/\PP})\to H^1(C, \cT_C)\to 0
\]
is exact, where the middle term can be regarded as the tangent space of  $\hilb (\PP)$ at $[C]$  (and has  dimension  $(g^2-1) +(3g-3)$). 
\end{lemma}
\begin{proof}
The Euler sequence, which describes the tangent sheaf $\cT_\PP$ of the projective space $\PP=\check\PP(H^0(C,\Omega_C))$, tensored with $\cO_C$ is 
\[
0\to\cO_C\to \Omega_C\otimes H^0(C,\Omega_C)^\ast\to  \cO_C\otimes \cT_\PP\to 0.
\]
Consider its  cohomology exact sequence:
\begin{multline*}
0\to H^0(C, \cO_C)\to \End (H^0(C,\Omega_C))\to H^0(C, \cO_C\otimes \cT_\PP)\to\\
\to H^1(C, \cO_C)\to \hom (H^0(C,\Omega_C), H^1(C,\Omega_C))\to H^1(C,  \cO_C\otimes \cT_\PP)\to 0.
\end{multline*}
The first nonzero map of the first line is just the inclusion of the scalars in $\End (H^0(C,\Omega_C))$  (whose cokernel is $\autfrak(\PP)$) and the 
first map of the second line is readily verified to be the isomorphism provided by (Serre) duality. So $\autfrak(\PP)\to H^0(C,  \cO_C\otimes \cT_\PP)$ is an isomorphism and  $H^1(C,  \cO_C\otimes \cT_\PP)=0$. If we use the last observation as input for the exact cohomology sequence of  the short exact sequence
\[
0\to \cT_C\to \cO_C\otimes \cT_\PP\to \cN_{C/\PP}\to 0,
\]
we obtain the stated exact sequence  (where we use that $H^0(C, \cT_C)=0$).
\end{proof}

\begin{lemma}\label{lemma:sheafD}
We have a natural isomorphism  $\cO_C\otimes \cN_{X/\PP}\cong \Omega^{\otimes 2}_C$ and an exact sequence 
\[0\to H^0(C,\cN_{C/X})\to H^0(C,\cN_{C/\PP})\xrightarrow{\phi_X} H^0(C, \Omega^{\otimes 2}_C)\to \ext_C^1(\cC_{C/X}, \cO_C)\to 0.
\]
\end{lemma}

\begin{proof}
Consider the standard exact sequence of conormal sheaves associated to 
the chain of embeddings $C\subset X\subset\PP$:
\[
\cO_C\otimes\cC_{X/\PP}\to \cC_{C/\PP}\sr{p}{\to}\cC_{C/X}\to 0.
\]
Since $X$ is a hypersurface in $\PP$ of degree 2, the embedding 
$X\subset\PP$ is regular with conormal sheaf $\cC_{X/\PP}$ isomorphic to $\cO_X(-2)$. Recalling that  $\cO_C(1)=\Omega_C$, this yields an isomorphism $\cO_C\otimes \cC_{X/\PP}\cong (\Omega_C^{\otimes 2})^\vee$. 
Now $\cC_{C/\PP}$ is  the conormal sheaf of a regular embedding and so locally free. 
It follows that $\ker p$ is torsion free  and (hence)  locally free of rank one. So  $\cO_C\otimes \cC_{X/\PP}\cong (\Omega_C^{\otimes 2})^\vee\to\ker p$, being a surjective morphism of invertible sheaves, must be  an isomorphism. 
We thus obtain a locally free resolution of $\cC_{C/X}$:
\[
0\to (\Omega_C^{\otimes 2})^\vee\to\cC_{C/\PP}\sr{p}{\to}\cC_{C/X}\to 0.
\]
Applying  $\hom_C(-,\cO_C)$ to this resolution, gives the exact sequence
\begin{gather*}
0\to H^0(C,\cN_{C/X})\to H^0(C,\cN_{C/\PP})\to H^0(C, \Omega^{\otimes 2}_C)\to \ext_C^1(\cC_{C/X}, \cO_C)\to 
H^1(C, \cN_{C/\PP}).
\end{gather*}
The exact sequence of the lemma now follows from the vanishing of $H^1(C, \cN_{C/\PP})$. 
\end{proof}

\begin{corollary}[=Part (iii) of Theorem \ref{obstruction2}]\label{cor:van}
Assume that $X$ is nonsingular. Then the cohomology spaces $H^1(C,\cO_C\otimes\cT_{X})$ and $\ext_C^1(\cC_{C/X}, \cO_C)$ vanish.
In particular, the deformations of $C$ in $X$ are unobstructed and $\hilb(X)$ is nonsingular of dimension 
$\dim H^0(C,\cN_{C/\PP})-\dim H^0(C, \Omega^{\otimes 2}_C)=g^2-1$ at $[C\subset X]$. 
\end{corollary}

\begin{proof}
We prove that the restriction of  $\phi_X$ to the subspace $\autfrak(\PP)\subset  H^0(C,\cN_{C/\PP})$ is onto;
this will clearly imply the corollary. This map is defined as follows: regard $A\in \autfrak(\PP)$ as a vector field
on $\PP$, restrict it to $X$ so that we get a normal vector field to $X$ in $\PP$ and then restrict this normal vector field
to $C$ and take its image in $H^0(C, \cO_C\otimes \cN_{X/\PP})\cong H^0(C, \Omega^{\otimes 2}_C)$.
This map factors through $\sym^2 H^0(C,\Omega_C)\to H^0(C,\Omega^{\otimes 2}_C)$: if we identify $\autfrak(\PP)$ with 
$\slfrak(H^0(C, \Omega_C))\subset \End (H^0(C, \Omega_C))$, then the lift in question is given by 
\[
 A\in \End (H^0(C, \Omega_C))\mapsto  (A\otimes 1 +1\otimes A)(Q)\in \sym^2 H^0(C, \Omega_C).
\]
Since $Q$ is nonsingular, the above map is onto and following Max Noether, the same is true for the map 
$\sym^2 H^0(C,\Omega_C)\to H^0(C,\Omega^{\otimes 2}_C)$. Now $\End (H^0(C, \Omega_C))$ is the direct sum of $\slfrak(H^0(C, \Omega_C))$ 
and the scalars. But the scalars map under the above map to the multiples of $Q$ and hence vanish when restricted to $C$. It follows that 
$\phi_X|\autfrak(\PP)$ is onto as asserted.
\end{proof}

We next consider the general case. An obstruction theory for the embedding $C\subset X$ is given by the vector space $\ext^1_C(\cC_{C/X},\cO_C)$ 
(cf.\ Proposition~2.14 in \cite{Kollar}), but it is not always true that  this space consists entirely of obstructions, i.e., 
is a  \emph{minimal obstruction space} in the sense of Definition~5.5 in \cite{T-V}. Our main theorem states however that in our situation this is so. In order to show this, we need some preliminary results on Hilbert schemes of canonical curves in quadrics.
The Hilbert scheme of quadrics in $\PP$ is naturally identified with $\PP(\sym^2 H^0(C, \Omega_C))$. It comes with a universal family of quadrics:
\[
\begin{array}{ccc}
{\mathcal U}&\subset&\PP\times\PP(\sym^2 H^0(C, \Omega_C))\\
&\searrow&\da\\
&&\PP(\sym^2 H^0(C, \Omega_C)).
\end{array}
\]
Let $S\subset\PP(\sym^2 H^0(C, \Omega_C))$ be the open subscheme parameterizing quadrics of rank $\geq 3$ and ${\mathcal U}_S\to S$ the restriction of the universal 
family over $S$. Denote by  $\hilb^\circ({\mathcal U}_S/S)$ the subscheme of the relative Hilbert scheme $\hilb({\mathcal U}_S/S)$ 
whose closed points parameterize the pairs consisting of  a nonsingular canonically embedded genus $g$ curve in $\PP$  and 
a quadric of rank $\geq 3$ in $\PP$  containing that curve. Note that $\hilb^\circ({\mathcal U}_S/S)$ is an open subscheme of $\hilb({\mathcal U}_S/S)$ and  which comes with a surjective morphism $p\co\hilb^\circ({\mathcal U}_S/S)\to S$.

\begin{lemma}\label{lemma:syntomic}
The scheme $\hilb^\circ({\mathcal U}_S/S)$ is irreducible. Moreover, $p\co\hilb^\circ({\mathcal U}_S/S)\to S$ is a syntomic 
(i.e., a flat local complete intersection) morphism of relative dimension $g^2-1$ that is generically smooth. In particular, $\hilb^\circ({\mathcal U}_S/S)$ 
is of dimension $\dim(S)+ g^2-1$. 
\end{lemma}

\begin{proof}Let $\hilb^\circ(\PP)$ be the subscheme of $\hilb(\PP)$ parameterizing nonsingular 
canonically embedded projective genus $g$ curves in $\PP$. It is well known that $\hilb^\circ(\PP)$ is a smooth open subscheme of 
an irreducible component of $\hilb(\PP)$. Since $\hilb(\PP\times S/S)\cong\hilb(\PP)\times S$, the product
$\hilb^\circ(\PP)\times S$ is identified with a smooth open subscheme of an irreducible component of $\hilb(\PP\times S/S)$.

The $S$-embedding ${\mathcal U}_S\subset\PP\times S$ induces a morphism 
\[
\hilb^\circ({\mathcal U}_S/S)\to\hilb^\circ(\PP)\times S\subset\hilb(\PP\times S/S).
\]
The morphism  $\pi\co\hilb^\circ({\mathcal U}_S/S)\to\hilb^\circ(\PP)$ obtained by  composition with the projection on $\hilb^\circ(\PP)$ 
is a surjection, the fiber of $p$ over a closed point $[C\subset\PP]\in\hilb(\PP)$ being 
naturally identified with the linear system of quadrics in $\PP$ containing $C$ (which consists of quadrics of rank $\geq 3$). Therefore it
is irreducible and hence $\hilb^\circ({\mathcal U}_S/S)$ is irreducible as well. 

We claim that a general nonsingular nonhyperelliptic canonically embedded 
curve $C\subset\PP$ of genus $\geq 4$ is contained in a nonsingular quadric. For $g=4$, this follows from the description of canonical curves in $\PP^3$ as complete intersections of a quadric with an irreducible cubic surface, whereas  Petri's and Bertini's theorems imply that for a nonsingular  projective curve $C$ of genus $\ge 5$, which is neither trigonal nor isomorphic to a plane quintic, the general member of the linear system of quadrics of $\PP$ containing $C$ is nonsingular: indeed, by Petri's Theorem, 
the linear system of quadrics containing $C$ has then for base locus the curve $C$ and is without fixed components; by Bertini's Theorem,
a general member of this linear system then has no singularities outside of C. But, then again by Petri's Theorem, quadrics
generate the canonical ideal and so a general member of this system has no singularities on $C$ either. Therefore a general quadric of $\PP$ containing $C$ is nonsingular.

By Corollary~\ref{cor:van} and Proposition 4.4.7 in \cite{Sernesi}, the restriction of the morphism $p$ over the locus of $S$ parameterizing 
nonsingular quadrics is therefore smooth of relative dimension $h^0(C,\cN_{C/X})=g^2-1$.

In order to complete the proof of the lemma, it is enough to show that, given a closed point $[X]\in S$ and a closed point 
$[C\subset X]\in\hilb^\circ({\mathcal U}_S/S)$ over it, the morphism $p\co\hilb^\circ({\mathcal U}_S/S)\to S$ is
syntomic at $[C\subset X]$. Since $\hilb^\circ({\mathcal U}_S/S)$ is an open subscheme of $\hilb({\mathcal U}_S/S)$, we have:
\begin{equation}
\dim_{[C\subset X]}\hilb({\mathcal U}_S/S)=g^2-1+\dim S.
\end{equation}
By (ii) of Theorem~2.15 in \cite{Kollar}, there is an inequality:
\begin{equation}
\dim_{[C\subset X]}\hilb({\mathcal U}_S/S)\geq h^0(C,\cN_{C/X})+\dim S-\dim\ext^1_C(\cC_{C/X},\cO_C).
\end{equation}
The exact sequence in Lemma~\ref{lemma:sheafD} gives the identity:
\begin{multline}
h^0(C,\cN_{C/X})=h^0(C,\cN_{C/\PP})-h^0(C,\Omega^{\otimes 2}_C)+\dim\ext^1_C(\cC_{C/X},\cO_C)=\\
=g^2-1+\dim\ext^1_C(\cC_{C/X},\cO_C).
\end{multline}
Combining the identities $(1)$ and $(3)$, we get that the inequality $(2)$ is actually an identity.
By (iv) of Theorem~2.15 in \cite{Kollar}, this implies that $p\co\hilb^\circ({\mathcal U}_S/S)\to S$ is a syntomic
morphism at $[C\subset X]$.
\end{proof}

\begin{proof}[Proof of Theorem \ref{obstruction2}] It remains to prove parts (i) and (ii). An open neighborhood of the point
$[C\subset X]$ in the Hilbert scheme $\hilb (X)$ is naturally isomorphic to the fiber of the morphism $p\co\hilb^\circ({\mathcal U}_S/S)\to S$ over 
the point $[X]\in S$. Since the fibers of a syntomic morphism are local complete intersections,  Lemma~\ref{lemma:syntomic} implies that 
$\hilb (X)$ is a local complete intersection at the point $[C\subset X]$ of dimension $g^2-1$.

Let us denote by $\mathrm{Obs}(C/X)$ the minimal obstruction space in $\ext^1_C(\cC_{C/X},\cO_C)$ for deformations of $C$ in $X$.
By (iv) of Theorem~2.8 in \cite{Kollar} and the identity $(3)$, there is an inequality
\[\begin{array}{ll}
\dim_{[C\subset X]}\hilb (X)=g^2-1&\geq h^0(C,\cN_{C/X})-\dim\mathrm{Obs}(C/X)=\\
&=g^2-1+\dim\ext^1_C(\cC_{C/X},\cO_C)-\dim\mathrm{Obs}(C/X),
\end{array}\]
which  shows that $\dim\ext^1_C(\cC_{C/X},\cO_C)-\dim\mathrm{Obs}(C/X)=0$.  Thus the theorem follows.
\end{proof}

\begin{remark}\label{relative}Let $\mathrm{Obs}_S(C/{\mathcal U}_S)$ be the minimal obstruction space in $\ext^1_C(\cC_{C/X},\cO_C)$ for deformations of 
$C$ in ${\mathcal U}_S/S$. A corollary of the proof of Lemma~\ref{lemma:syntomic} is that $\mathrm{Obs}_S(C/{\mathcal U}_S)=\ext^1_C(\cC_{C/X},\cO_C)$.
Since $\ext^1_C(\cC_{C/X},\cO_C)$ is also the obstruction space of an obstruction theory for the embedding $C\subset X$,
it follows that there is a natural injective linear map $i\co\mathrm{Obs}(C/X)\hookra\mathrm{Obs}_S(C/{\mathcal U}_S)$. By the functorial property 
of obstruction theories, the map $i$ has a natural (not necessarily linear) section $s\co\mathrm{Obs}_S(C/{\mathcal U}_S)\to\mathrm{Obs}(C/X)$ induced 
by specializing to $[X]\in S$ the tiny extensions from which the elements of the minimal obstruction space $\mathrm{Obs}_S(C/{\mathcal U}_S)$ arise. 
We have just proved that $i$ is actually an isomorphism. However, this was not clear a priori. In fact, $\mathrm{Obs}_S(C/{\mathcal U}_S)$ could have 
contained obstructions coming from tiny extensions which became trivial when specialized to the point $[X]\in S$.
\end{remark}

By the local-global spectral sequence for $\ext$ and the vanishing of  $E_2^{2,0}=H^2(C,\cN_{C/X})$, there is a short exact sequence: 
\[0\to H^1(C,\cN_{C/X})\to\ext^1_C(\cC_{C/X},\cO_C)\to H^0(C,{\mathcal Ext}^1(\cC_{C/X},\cO_C))\to 0.\]

\begin{lemma}\label{lemma:sheafext}
Let $D$ be the (effective)  divisor on $C$ defined by the ideal defining  the singular locus of $X$. Then there is a short exact sequence
\[0\to\cN_{C/X}\to\cN_{C/\PP}\to\Omega_C^{\otimes 2}(-D)\to 0\]
and the sheaf ${\mathcal Ext}^1_C(\cC_{C/X},\cO_C)$ can be canonically identified with the direct image on $C$ of the skyscraper sheaf 
$\cO_D\otimes\Omega_C^{\otimes 2}$.
In particular, $\dim H^0(C,{\mathcal Ext}^1(\cC_{C/X},\cO_C))=\deg D$.
\end{lemma}

\begin{proof}In the proof of Lemma~\ref{lemma:sheafD}, we obtained the locally free resolution of $\cC_{C/X}$:
\[
0\to (\Omega_C^{\otimes 2})^\vee\to\cC_{C/\PP}\sr{p}{\to}\cC_{C/X}\to 0.
\]
The  map $\cO_C\otimes\cC_{X/\PP}\cong(\Omega_C^{\otimes 2})^\vee\to\cC_{C/\PP}$ is best understood in terms of affine local coordinates.
Let $U$ be a standard affine open subset of $\PP$ and $q$ a generator of the ideal of the affine variety $U\cap X$. Then $q$ is a local generator
for the sheaf $\cO_C\otimes\cC_{X/\PP}$ on $U\cap C$ and $dq$ is a differential on $U$ whose restriction to $U\cap C$ vanishes on the tangent fields 
to $U\cap C$ and then defines an element of the conormal sheaf $\cI_C/\cI^2_C=\cC_{C/X}$. The restriction of $dq$ to $C$ is just the image of $q$ 
under the map $\cO_C\otimes\cC_{X/\PP}\to\cC_{C/\PP}$.

Since the partial derivatives of $q$, with respect to a system of affine coordinates in $U$, define the ideal of the singular locus of $X$, the
$\cO_C$-dual $\cN_{C/\PP}\to\Omega_C^{\otimes 2}$ of the above map has  $\Omega_C^{\otimes 2}(-D)$ as image and hence $\cO_D\otimes\Omega_C^{\otimes 2}$ as cokernel. This proves all the statements of the lemma.
\end{proof}

An immediate consequence of Theorem~\ref{obstruction2} and Lemma~\ref{lemma:sheafext} is then:

\begin{corollary}\label{singular}
Under the hypotheses of Theorem~\ref{obstruction2}, $\dim\mathrm{Obs}(C/X)\geq\deg D$. So if
the deformations of $C$ in $X$ are unobstructed, then $C$ does not meet the singular locus of $X$.
\end{corollary}

\begin{question}Are the deformations of $C$ in $X$ unobstructed \emph{if and only if} the curve $C$ does not meet the singular locus of the quadric $X$? In other words, does the last property imply the vanishing of $H^1(C, \cN_{C/X})$?
\end{question}

\end{document}